\documentclass{article}   	
\usepackage{geometry}               
\usepackage{graphicx}
\usepackage[usenames,dvipsnames]{xcolor}
\usepackage{amssymb,amsmath,amsthm,amsfonts}

\usepackage[shortlabels]{enumitem}

\usepackage{hyperref} 
\usepackage[capitalise, nameinlink, noabbrev]{cleveref} 
\hypersetup{colorlinks=true, 
	linkcolor=blue,
	citecolor=red,
}
\allowdisplaybreaks

\usepackage[english]{babel}
\usepackage[utf8]{inputenc}

\usepackage{mathtools}
\usepackage{esint} 
\usepackage{dsfont}

\usepackage[sc]{mathpazo}

\usepackage{tikz}
\usepackage{xfrac}

\newtheorem{proposition}{Proposition}[section]
\newtheorem{theorem}[proposition]{Theorem}
\newtheorem{corollary}[proposition]{Corollary}
\newtheorem{lemma}[proposition]{Lemma}

\newtheorem{assumptions}[proposition]{Assumptions}
\theoremstyle{definition}

\theoremstyle{remark}
\newtheorem{remark}[proposition]{Remark}
\numberwithin{equation}{section}

\newcommand{\eps}{\varepsilon}

\newcommand{\N}{{\mathbb{N}}}

\newcommand{\R}{{\mathbb{R}}}

\DeclareMathOperator{\diverg}{div}

\title{Unique continuation for nonlinear variational problems}
\author{Lorenzo Ferreri, Luca Spolaor, Bozhidar Velichkov}

\begin{document}
\maketitle

\begin{abstract}
This paper is dedicated to the unique continuation properties of the solutions to nonlinear variational problems. Our analysis covers the case of nonlinear autonomous functionals depending on the gradient, as well as more general double phase and multiphase functionals with $(2,q)$-growth in the gradient. We show that all these cases fall in a class of nonlinear functionals for which we are able to prove weak and strong unique continuation via the almost-monotonicity of Almgren's frequency formula. As a consequence, we obtain estimates on the dimension of the set of points at which both the solution and its gradient vanish. 
        \\%
\end{abstract} 
\noindent
{\footnotesize \textbf{AMS-Subject Classification}}. 
{\footnotesize 35B40, 35J60, 
}

\tableofcontents

\section{Introduction}
Unique continuation type results for elliptic operator have been a central theme of investigation in PDEs for many years, we refer for instance to  \cite{Ar,JeKe, GarofaloLin1987:UniqueContinuationFrequency, SoaveTerracini2018:UniqueContinuationSublinear, Logunov2018:NodalSetsLaplaceEigenfunctions, FallFelli2014:UniqueCOntinuationFractional, Yu2017:UniqueContinuationFractional} for the cases of linear and semilinear operators, the analysis of eigenfunctions and the fractional case. In this paper, on the other hand, we present a result on the (strong) unique continuation property for nonlinear elliptic equations. 

As an introductory example, consider the functional $\mathcal{F}: W^{1, q}(B) \to \R$ defined by
\begin{equation}\label{eqn:quasilinearFunctionalF}
\mathcal{F}(u) \coloneqq \int_B L(\nabla u)\,,
\end{equation}
where the lagrangian $L: \R^d \to \R$ is 
\begin{equation}\label{eqn:QNDoublePhaseFunctional}
L(p) \coloneqq \frac{1}{2} \vert p \vert^2 + \frac{1}{q} \vert p \vert^q, \quad\text{with}\quad q >2,
\end{equation}
and the associated variational problem is
\begin{equation}\label{eqn:quasilinearVariationalPb}
\text{argmin} \left\{ \mathcal{F}(\varphi): \varphi \in W^{1, q}(B) \text{ and } \varphi = u_0 \text{ on } \partial B \right\}
\end{equation}
for some boundary datum $u_0 \in W^{1, q}(B)$ and some ball $B\subset\R^d$.

Let the function $u$ be the unique (by strict convexity) solution to \eqref{eqn:quasilinearVariationalPb}, then $u$ solves the nonlinear elliptic equation
\begin{equation}\label{eqn:minimizerUEqn}
\diverg\left( (1+\vert \nabla u \vert^{q-2}) \nabla u \right)=0.
\end{equation}
Thanks to the pioneering work of Marcellini \cite{Marcellini89:RegularityMinimizersNonStandardGrowth}, it is now known that if $q$ satisfies some upper bounds depending on the dimension (see \cref{sub:examples} below), the solution $u$ is at least $C^{2, \alpha}$ smooth (\cite[Theorem E]{Marcellini89:RegularityMinimizersNonStandardGrowth}). We notice that when $2<q<3$, the function $u$ is a solution to a problem of the form 
\[
\diverg\left(A(x)\nabla u\right)=0,
\]
where the matrix field
\begin{equation}\label{eqn:matrixFieldA(x)Def}
A(x):=(1+\vert \nabla u \vert^{q-2}) Id
\end{equation}
is only $(q-2)$-H\"older continuous even when $ u\in C^{2}(B)$, provided that $\nabla^2 u(x) \neq 0$.    
Hence, to the best of our knowledge, the well established theory of Garofalo and Lin \cite{GarofaloLin1986:MonotinicityApWeights, GarofaloLin1987:UniqueContinuationFrequency}, which requires $A\in C^{0,1}$, cannot be employed to study the property of unique continuation for the solution $u$. Thus, the main obstruction to the development of a unique continuation theory for this type of functionals is not the possible lack of regularity of $u$, but the nonlinear nature of the lagrangian. \medskip

In this paper, using the strategy developed in \cite{FerreriSpolaorVelichkov2024:BoundaryBranchingOnePhaseBernoulli}, which in turn was inspired by \cite{Almgren2000:AlmgrensBigRegularityPaper, DeLellisSpadaro2016:AreaMinimizingCurrents1LpEstimates, DeLellisSpadaro2016:AreaMinimizingCurrents2CenterManifold, DeLellisSpadaro2016:AreaMinimizingCurrents3BlowUp}, we are able to prove a unique continuation result for a class of nonlinear problems including the above case. In particular, even though the matrix field $A(x)$ from \eqref{eqn:matrixFieldA(x)Def} is, in general, only H\"older continuous, the quasilinear structure of \eqref{eqn:minimizerUEqn} allows us to recover the strong unique continuation property. In fact, our strategy applies to more general functionals 
\begin{equation}\label{eqn:quasilinearFunctionalF-general}
\mathcal{F}: W^{1, q}(B) \to \R\ ,\qquad \mathcal{F}(\varphi) \coloneqq \int_B L(x, \varphi,\nabla \varphi),
\end{equation}
where the lagrangian $L: \R^d\times \R\times \R^d \to \R$ is 
\[
L(x, s,p) \coloneqq \frac{1}{2} \vert p \vert^2 + F(x,s,p),
\]
and where the function $F$ satisfies the following conditions:
\begin{itemize}
\item $F$ is Lipschitz continuous in $x$ and there are constants $\gamma, C>0$, and a neighborhood $\mathcal U\subset \R^d\times\R\times\R^d$ of the origin such that 
\begin{equation}\label{e:stima-grad-F-in-X}
|F(x,s,p)|+|\nabla_xF(x,s,p)|\le C\left(|p|^{2+\gamma}+|s|^{2}\right)\quad \text{ for every }\quad (x,s,p)\in\mathcal U;
\end{equation}

\item $F$ is differentiable in $p$ and there exist constants $\gamma, C>0$, and a neighborhood $\mathcal U\subset \R^d\times\R\times\R^d$ of the origin such that 
\begin{equation}\label{e:stima-grad-F-in-p}
|\nabla_pF(x,s,p)|\le C\left(|p|^{1+\gamma}+|s|^{1+\gamma}\right)\quad \text{ for every }\quad (x,s,p)\in\mathcal U\,;
\end{equation}

\item there exist constants $\gamma, C>0$ and a neighborhood $\mathcal U\subset \R^d\times\R\times\R^d$ of the origin such that, 
for every $(x,s,p)\in\mathcal U$, the function 
$$f_{x,s,p}(t):= F(p,s(1+t))$$
is differentiable at $t=0$ and we have the following estimate: 
\begin{equation}\label{e:stima-derivata-F-in-s}
|f_{x,s,p}'(0)|\le C\left(|p|^{2+\gamma}+|s|^{2}\right) \text{ for every } (x,s,p)\in\mathcal U\,.
\end{equation}
We notice that $F$ when differentiable at $s$, we have that 
$$f_{x,s,p}'(0)=s\,\partial_sF(x,s,p).$$
From now on, for simplicity, we will often write $s\,\partial_sF(x,s,p)$ in place of $f_{x,s,p}'(0)$.
\end{itemize}
To state our results more clearly, it is convenient to work under the following
\begin{assumptions}
    \label{assumptions:QNuC1aAprioriEst}
    There exist constants $0<\alpha, \delta_0<1$ such that
    \[
    \| u \|_{C^{1, \alpha}\left( B \right)} \le \delta \quad \text{for some } \delta\in(0,\delta_0).
    \]
    Moreover, there exists a constant $C = C(d, \delta_0)>0$ such that
    \[
    \| u \|_{C^{0, \alpha}\left( B_{r/2}(x) \right)} \le C \left( \fint_{B_r(x)} u^2 \right) \quad \text{and} \quad \| \nabla u \|_{C^{0, \alpha}\left( B_{r/2}(x) \right)} \le \frac{C}{r} \left( \fint_{B_r(x)} u^2 \right)
    \]
    for all balls $B_r(x)$ with $B_r(x) \subset B$.
\end{assumptions}
\begin{remark}[On the $C^{1,\alpha}$ regularity in \cref{assumptions:QNuC1aAprioriEst}]
The regularity of the minimizers of $\mathcal F$ is now known to be strongly related to the growth of $\mathcal F$ in the $p$ variable. The unique continuation, on the other hand, relies on the behavior of $\mathcal F$ near singular points of $u$, where $s=0$ and $p=0$.
\end{remark}
\begin{remark}[On the linear $C^{0,\alpha}$ estimates in \cref{assumptions:QNuC1aAprioriEst}]
Since the unique continuation property is localized at singular points, i.e. where $u=\vert \nabla u \vert = 0$, the linear estimates in \cref{assumptions:QNuC1aAprioriEst} are a consequence of the $C^{1, \alpha}$ regularity of $u$. For the validity of \cref{assumptions:QNuC1aAprioriEst} we refer to \cref{sub:examples} and \cref{section:C1aL2LinearEst} below.    
\end{remark}
Our main result is the following
\begin{theorem}[Strong unique continuation]\label{theorem:quasilinearUniqueContinuation}
    Let the function $u \in W^{1, q}(B)$ be a local minimizer of \eqref{eqn:quasilinearFunctionalF-general} under \cref{assumptions:QNuC1aAprioriEst} and suppose that for some point $x_0 \in B$
    \begin{equation}\label{eqn:QNvanishingAllOrdersHp}
    \lim_{r\to 0}\frac{1}{r^n}\fint_{B_r(x_0)} u^2 = 0 \qquad \text{for all}\quad n \in \N\,.
    \end{equation}
    Then,
    \[
    u \equiv 0\quad \text{in}\quad B.
    \]
\end{theorem}
\begin{corollary}[Weak unique continuation]\label{cor:intro-weak-unique-continuation}
    Let the function $u \in W^{1, q}(B) \cap C^{1, \alpha}(B)$ be a local minimizer of \eqref{eqn:quasilinearFunctionalF-general} and suppose that $u= 0$ in some open subset $U\subset B$, then $u\equiv 0$ in $B$.
\end{corollary}
Finally, as a consequence of the classical Federer's dimension reduction principle, we obtain the following 
\begin{theorem}[Dimension of the critical set]\label{teo:intro-dimension-critical-sets}
 Let the function $u \in W^{1, q}(B) \cap C^{1, \alpha}(B)$ be a local minimizer of \eqref{eqn:quasilinearFunctionalF-general}, then either $u\equiv 0$ in $B$ or
\[
\dim(\{u=0\,\text{ and }|\nabla u|=0\}\cap B)\leq d-2
\]    
\end{theorem}
%
The main idea to prove \cref{theorem:quasilinearUniqueContinuation} and \cref{teo:intro-dimension-critical-sets} is to show an (almost-)monotonicity of the frequency function for harmonic function by using ideas similar to \cite{FerreriSpolaorVelichkov2024:BoundaryBranchingOnePhaseBernoulli}. We remark that the same strategy applies also to more general energies, which can also depend less regularly on the variable $x$. For instance, 
\[
L(x,s,p) \coloneqq \frac{1}{2} (p\cdot A(x)p) + s b(x)\cdot p+V(x)\,s^2+F(x, s,p),
\]
with $A,b$ Lipschitz continuous and  $V\in L^\infty$.
We chose to present this paper in the simpler setting above, but for the required modifications one can look for instance at \cite{GarofaloLin1987:UniqueContinuationFrequency} or \cite{FerreriSpolaorVelichkov2024:BoundaryBranchingOnePhaseBernoulli}.

\subsection{Further examples}\label{sub:examples}

\begin{itemize}
    \item General \emph{autonomous functionals} with $(2,q)$ - growth (\cite{Marcellini89:RegularityMinimizersNonStandardGrowth})
    \[
    L(x,s,p)=\phi(p),
    \]
    where $\phi:\R^d\to\R$ is a $C^{2,\alpha}$ function for some $\alpha>0$. A ($C^{2,\alpha}$-)regularity theorem for minimizers of the above functionals was proved in \cite[Theorem E]{Marcellini89:RegularityMinimizersNonStandardGrowth} and \cite[Theorem 3 and Corollary 1]{BellaShaffner2020:RegularityMinimizersPQgrowth} for $\phi$ satisfying the following growth conditions:
   \begin{equation*}
   \begin{cases}
   m|p|^2\le \phi(p)\le M(1+|p|^q)\,,\\
   m|\xi|^2\le \xi\cdot\nabla^2\phi(p)\xi\le M(1+|p|^2)^{\frac{q-2}{2}}|\xi|^2\quad\text{for every}\quad\xi\in\R^d\,,
   \end{cases}
   \end{equation*}
    for some constants $0<m\le M<+\infty$, and under the following bounds on $q$:
   \begin{equation}\label{e:intro-conditions-bella-schaffner-q}
   2\le q\le 2+\min\left\{2,\frac{4}{d-1}\right\}.
   \end{equation}
Our unique continuation theorems (\cref{theorem:quasilinearUniqueContinuation}, \cref{cor:intro-weak-unique-continuation} and \cref{teo:intro-dimension-critical-sets}) apply to functionals with lagrangians of the form 
$$L(x,s,p)=\frac12|p|^2+\psi(p),$$
where $\psi:\R^d\to\R$, $d\ge 2$, is a $C^{2,\alpha}$-regular non-negative convex function satisfying 
\begin{equation}\label{e:intro-our-conditions-on-bella-schaffner-1}
\begin{cases}
   0\le \psi(p)\le M|p|^q\,,\\
  0\le \xi\cdot\nabla^2\psi(p)\xi\le M|p|^{q-2}|\xi|^2\quad\text{for every}\quad\xi\in\R^d\,,
   \end{cases}
 \end{equation}  
 where $M>0$ is a positive constant and where the exponent $q$ satisfies the bounds
 \begin{equation}\label{e:intro-our-conditions-on-bella-schaffner-2}
   2< q\le 2+\min\left\{2,\frac{4}{d-1}\right\}\,.
   \end{equation}
 Indeed, the conditions \eqref{e:intro-our-conditions-on-bella-schaffner-1} and \eqref{e:intro-our-conditions-on-bella-schaffner-2} on $\psi$ assure that the lagrangian satisfies both the conditions from \cite{BellaShaffner2020:RegularityMinimizersPQgrowth,Marcellini89:RegularityMinimizersNonStandardGrowth} (so the solutions are $C^{1,\alpha}$ regular) and the conditions \eqref{e:stima-grad-F-in-X} and \eqref{e:stima-grad-F-in-p}. Finally, for this functional, the $C^{1,\alpha}$ regularity of $u$ implies that for small $\delta$ the linear estimates from \cref{assumptions:QNuC1aAprioriEst} hold (see \cref{prop:C1a-L2LinearBounds}).
   \item \emph{Double phase functionals}. Consider the lagrangian 
    \[
    L(x,s,p)=|p|^2+a(x)|p|^q
    \]
    where the coefficient $a$ is non-negative and Lipschitz continuous, and where $q$ satisfies the condition 
\begin{equation}\label{e:double-phase-q-exponent-condition}
 2<q\le2+\frac{2}{d}\,.   
\end{equation}
Then, the unique continuation theorems \cref{theorem:quasilinearUniqueContinuation}, \cref{cor:intro-weak-unique-continuation} and \cref{teo:intro-dimension-critical-sets} hold for any minimizer $u$ of $\mathcal F$. Indeed, under the condition \eqref{e:double-phase-q-exponent-condition}, the minimizers $u$ of $\mathcal F$ are in $C^{1,\alpha}$ for some $\alpha>0$ (see \cite{BaroniMingioneColombo2018CalcVar,ColomboMingione2015ARMA}). At the same time, it is immediate to check that the conditions \eqref{e:stima-grad-F-in-X} and \eqref{e:stima-grad-F-in-p} are fulfilled. Finally, as in the case of autonomous functionals, the linear estimates from \cref{assumptions:QNuC1aAprioriEst} hold as a consequence of \cref{prop:C1a-L2LinearBounds}.
    \item \emph{Multiphase functionals.} Consider the functional
    \[
    L(x,s,p)=|p|^2+a(x)|p|^{q}+b(x)|p|^s,
    \]
    where $a$ and $b$ are non-negative Lipschitz functions, and where the exponents $q$ and $s$ satisfy the condition
    \begin{equation}\label{e:intro-multi-phase-q-exponent-condition}
 2<q\le s\le2+\frac{2}{d}\,.   
\end{equation}
Then, the conclusions of \cref{theorem:quasilinearUniqueContinuation}, \cref{cor:intro-weak-unique-continuation} and \cref{teo:intro-dimension-critical-sets} hold for any minimizer $u$ of $\mathcal F$ of this form. Indeed, in \cite{DeFilippisOh:RegularityMultiPhaseVariationalPb} it was shown that the condition \eqref{e:intro-multi-phase-q-exponent-condition}, together with the Lipschitz continuity and the positivity of $a$ and $b$, implies the $C^{1,\alpha}$ regularity of the minimizers $u$ to $\mathcal F$. Thus, in a neighborhood of a point $x_0$ such that $u(x_0)=|\nabla u(x_0)|=0$ the $C^{1,\alpha}$ norm of $u$ is small and, by Proposition \cref{prop:C1a-L2LinearBounds}, \cref{assumptions:QNuC1aAprioriEst} is fulfilled. Finally, it is immediate to check that \eqref{e:stima-grad-F-in-X} and \eqref{e:stima-grad-F-in-p} hold for this functional, so we can apply \cref{theorem:quasilinearUniqueContinuation}, \cref{cor:intro-weak-unique-continuation} and \cref{teo:intro-dimension-critical-sets}.
\end{itemize}

\section{Linear $C^{1, \alpha}-L^2$ estimates}\label{section:C1aL2LinearEst}
In this section we show how a $C^{1, \alpha}$ regularity assumptions in $B$ for minimizers of \eqref{eqn:quasilinearFunctionalF-general} can be used to prove linear $C^{1, \alpha}-L^2$ near singular points, i.e. where $u=0$ and $\nabla u = 0$. In particular, since in all the examples presented in \cref{sub:examples} the lagrangian depends only on $x$ and $\nabla u$, for the sake of simplicity we only consider such dependence. 

More precisely, in this section we consider functionals $\mathcal{F}: W^{1, q}(B) \to \R$ of the form
\begin{equation}\label{eqn:quasilinearFunctionalF-x-p}
\mathcal{F}(\varphi) = \int_B L(x,\nabla \varphi(x))\,dx
\end{equation}
where the lagrangian $L: \R^d\times\R^d \to \R$ is of the form 
\[
L(x,p) \coloneqq \frac{1}{2} \vert p \vert^2 + F(x,p),
\]
and the function $F$ satisfies the hypotheses \eqref{e:stima-grad-F-in-X} and \eqref{e:stima-grad-F-in-p}.

We are ready for to state the following
\begin{proposition}\label{prop:C1a-L2LinearBounds}
    Let the function $u \in W^{1, q}(B) \cap C^{1, \alpha}(B)$ be a local minimizer of \eqref{eqn:quasilinearFunctionalF-general} and suppose that
    \begin{equation}\label{eqn:C1aUdelta}
    \| u \|_{C^{1, \alpha}(B)} \le \delta.
    \end{equation}
    Moreover, assume also the following two conditions:
    \begin{itemize}
        \item[(i)] $\nabla_p F(\cdot, p) \in C^{0, \alpha}(B)$ for some $\alpha>0$ and every $p \in B$, with the estimate
        \begin{equation}\label{eqn:furtherAssC0a}
             \left[\nabla_p F(\cdot, p) \right]_{C^{0, \alpha}(B)} = o(1) \quad \text{as } \delta \to 0.
        \end{equation}
        where the quantity $o(1)$ is uniform in $p$; 

        \item[(ii)] $\nabla_p F(x, \cdot) \in C^1(B)$ for all $x \in B$ and
        \begin{equation}\label{eqn:furtherAssC1p}
             \nabla_p^2 F(x, 0) = 0.
        \end{equation}
    \end{itemize}
    There exists a constant $\delta_0 = \delta_0(d) > 0$ such that, if $\delta \le \delta_0$ in \eqref{eqn:C1aUdelta}, then
    \begin{equation}\label{e:linear-bounds}
    \| u \|_{C^{0, \alpha}\left( B_{r/2}(x) \right)} \le C \left( \fint_{B_r(x)} u^2 \right) \quad \text{and} \quad \| \nabla u \|_{C^{0, \alpha}\left( B_{r/2}(x) \right)} \le \frac{C}{r} \left( \fint_{B_r(x)} u^2 \right)
    \end{equation}
    for some constant $C = C(d, \delta_0)>0$ and all balls $B_r(x)$ with $B_r(x) \subset B$.
\end{proposition}
\begin{remark}
    Notice that, both conditions (i) and (ii) in \cref{prop:C1a-L2LinearBounds} are satisfied by all the examples in \cref{sub:examples}.
\end{remark}
The proof of \cref{prop:C1a-L2LinearBounds} is divided in two steps:
\begin{enumerate}
    \item in \cref{lemma:C0aU-L2LinearBound} we show the linear $C^{0, \alpha}$ bound for $u$, using the classical De Giorgi-Nash-Moser iterations;

    \item in \cref{lemma:C0aGradU-L2LinearBound} we prove the linear $C^{0, \alpha}$ bound for $\nabla u$; this is carried out using step 1 and a linearization argument.
\end{enumerate}
Let us proceed in order.
\begin{lemma}\label{lemma:C0aU-L2LinearBound}
    Let $u$, $\delta$, $\delta_0$, $C$ and $B_r(x)$ be as in \cref{prop:C1a-L2LinearBounds}. Then, 
\begin{equation*}
    \| u \|_{C^{0, \alpha}\left( B_{r/2}(x) \right)} \le C \left( \fint_{B_r(x)} u^2 \right).
\end{equation*}
\end{lemma}
\begin{proof}
Testing the outer variation for \eqref{eqn:quasilinearFunctionalF-general} with the competitors
\[
\varphi^{\pm} \coloneqq \eta^2 \left( u - k \right)^{\pm}, \quad k \in \R,
\]
for some smooth cutoff function $\eta:B_r(x_0) \to \R$, with computations analogous to \cref{section:OuterAndInnerVariations} it is possible to choose $\delta_0 = \delta_0(d)>0$ sufficiently small so that, thanks to \eqref{e:stima-grad-F-in-p} and \eqref{e:stima-derivata-F-in-s}, the Caccioppoli inequalities on super/sublevel sets are satisfied:
\begin{align*}
\int_{ \{u > k\} \cap B_{\rho}(x_0)} \vert \nabla u \vert^2 \le \frac{C}{(r- \rho)^2} \int_{ \{u > k\} \cap B_r(x_0)} (u-k)^2, \\
\int_{ \{u < k\} \cap B_{\rho}(x_0)} \vert \nabla u \vert^2 \le \frac{C}{(r- \rho)^2} \int_{ \{u < k\} \cap B_r(x_0)} (u-k)^2.
\end{align*}
for some constant $C=C(d, \delta_0)>0$ and all balls $B_{\rho}(x_0) \Subset B_r(x_0) \Subset B$. The lemma now follows from the De Giorgi-Nash-Moser iterations.
\end{proof}
\begin{lemma}\label{lemma:C0aGradU-L2LinearBound}
    Let $u$, $\delta$, $\delta_0$, $C$ and $B_r(x)$ be as in \cref{prop:C1a-L2LinearBounds}. Moreover, suppose that for some linear function $l:\R^d \to \R$
    \begin{equation*}
        \left\| u - u(x) - l \right\|_{L^{\infty}(B_r(x))} \le r \varepsilon
    \end{equation*}
    and that
    \begin{equation}\label{eqn:QNC1aLinearxHp}
        \left\| \nabla_p F( \cdot, \nabla l) - \nabla_p F(x, \nabla l) \right\|_{L^{\infty}(B_r(x))} \le o(1) \varepsilon \quad \text{as } \delta \to 0
    \end{equation}
    Then, there exist constants $\delta_0 = \delta_0(d)$, $\varepsilon_0 = \varepsilon_0(d, \delta_0)>0$, $\rho=\rho(d, \delta_0) \in (0, 1)$, $\widetilde{C}=\widetilde{C}(d)>0$ and a linear function $l':\R^d \to \R$ such that 
    \begin{equation*}
        \| u(y) -  u(x) - l' \|_{L^{\infty}\left( B_{\rho r}(x) \right)} \le \widetilde{C} \rho^{1+\alpha} \| u(y) -  u(x) - l \|_{L^{\infty}\left( B_{r}(x) \right)}
    \end{equation*}
    and
    \begin{equation*}
        \| \nabla l' - \nabla l \| \le \widetilde{C} \varepsilon.
    \end{equation*}
\end{lemma}
\begin{proof}
 We proceed by contradiction in $\varepsilon$ and $\delta$, via a linearizaion argument.    

 Suppose that there exist a sequence of local minimizers $u_k$ to \eqref{eqn:quasilinearFunctionalF-x-p}, radii $r_k \to 0$, points $x_k \in B$, constants $\delta_k, \varepsilon_k \to  0$ and linear functions $l_k$ satisfying
 \begin{equation*}
        \left\| u_k - u_k(x_k) - l_k \right\|_{L^{\infty}(B_{r_k}(x_k))} = r_k \varepsilon_k
\end{equation*}
 but
    \begin{equation}\label{eqn:QNC1aContradLinear}
        \| u_k(y) -  u_k(x_k) - l_k' \|_{L^{\infty}\left( B_{\rho r_k}(x_k) \right)} \ge k \rho_k^{1+\alpha} \| u_k(y) -  u_k(x_k) - l_k \|_{L^{\infty}\left( B_{r_k}(x_k) \right)}
    \end{equation}
    for every $\rho \in (0, 1)$ and linear function $l'$ with 
    \begin{equation}\label{eqn:QNC1aContradNablalDist}
    \| \nabla l' - \nabla l_k \| \le c \varepsilon_k,    
    \end{equation}
    for some large but fixed constand $c=c(d)>0$.

Without loss of generality we can assume that $x_k=0$ and, by translation invariance, also that $u_k(x_k)=0$. Let us introduce the linearized functions $w_k:B \to \R$ defined as
 \[
 w_k(x) \coloneqq  \frac{u_k(r_k x ) - l_k(r_k x)}{r_k \varepsilon_k}.
 \]
 The functions $w_k$ solve
\begin{align}\label{eqn:QNC1aLinearrizationEqn}
    \begin{split}
        \int_{B} \nabla w_k \cdot \nabla \psi +  \frac{1}{\varepsilon_k} \int_{B} \left[ \nabla_p F(r_k x , \nabla l_k + \varepsilon_k \nabla w_k) - \nabla_p F(0, \nabla l_k)  \right]\cdot \nabla \psi = 0, 
    \end{split}
\end{align}
for all functions $\psi \in C_c^{\infty}(B)$, and we have the estimate
\begin{align*}
    \begin{split}
        & \left\vert \frac{1}{\varepsilon_k} \left[ \nabla_p F(r_k x, \nabla l_k + \varepsilon_k \nabla w_k) - \nabla_p F(0, \nabla l_k)  \right] \right\vert \le \\
        & \le  \frac{1}{\varepsilon_k} \left[ \left\vert \nabla_p F( r_k x, \nabla l_k) - \nabla_p F(0, \nabla l_k) \right\vert + \left\vert \nabla_p F\left( r_k x, \nabla l +  \varepsilon_k \nabla w_k \right) - \nabla_p F(r_k x, \nabla l_k) \right\vert \right] \\
        & \le o(1) + \left\| \nabla_p^2 F(x, \cdot) \right\|_{L^{\infty}\left(B_{\delta_0}\right)} \vert \nabla w_k(x) \vert \\
        & \le o(1) \left[ 1 + \vert \nabla w_k(x) \vert \right] \quad \text{as } \delta \to 0,
    \end{split}
\end{align*}
where we have used \eqref{eqn:QNC1aLinearxHp} and \eqref{eqn:furtherAssC1p}. In particular, the above estimate combined with \eqref{eqn:QNC1aLinearrizationEqn} implies the Caccioppoli inequality for the functions $w_k$. Hence, using also \cref{lemma:C0aU-L2LinearBound}, the linearizations $w_k$ converge weakly in $H^1(K)$ and in $C^{0, \alpha}(K)$, for all $K \Subset B$, to a limit function $w_{\infty}$ which is harmonic in $B$, and satisfies $w_{\infty}(0) = 0$ and $\| w_{\infty} \|_{L^{\infty}(B)} \le 1$. Consequently, the function $w_{\infty}$ satisfies
\begin{equation}\label{eqn:QNC1aLinearLimitHarmonicEst}
\| w_{\infty} - \nabla w_{\infty} \cdot x \|_{L^{\infty}(B_{\rho})} \le C \rho^{1+\alpha} \quad \text{for all } \rho \in (0, 1/2),
\end{equation}
where the constant $C=C(d)>0$.

Now, as a consequence of \eqref{eqn:QNC1aLinearLimitHarmonicEst} and the $C^{0, \alpha}$ convergence of $w_k$ to $w_{\infty}$, for $k$ sufficiently large there exist constants $C=C(d)>0$ and $\rho=\rho(d)$ such that
\begin{equation}\label{eqn:QNC1aEstContrad}
\left\| u_k(r_k x) - \left( \nabla l_k - \varepsilon_k \nabla w_{\infty}(0) \right) \cdot x \right\|_{L^{\infty}(B_{\rho})}  \le C \varepsilon_k \rho^{1+\alpha} \quad \text{in } B_{\rho}
\end{equation}
Since \eqref{eqn:QNC1aEstContrad} is in contradiction with \eqref{eqn:QNC1aContradLinear} as long as $\vert \nabla w_{\infty}(0) \vert \le c$ (where $c$ is the constant from \eqref{eqn:QNC1aContradNablalDist}), the proof is concluded.
\end{proof}
We are ready for the
\begin{proof}[Proof of \cref{prop:C1a-L2LinearBounds}.]
The linear bound for $u$ in \eqref{e:linear-bounds} follows directly from \cref{lemma:C0aU-L2LinearBound}.
In order to prove the linear bound on the gradient, we take any point $y\in B_{r/2}(x)$ and we choose $\delta$ in such a way that 
$$\left\| u - u(y) \right\|_{L^{\infty}(B_{r/4}(y))}\le \frac{r}{4}\eps_0,$$
where $\varepsilon_0$ is the constant from \cref{lemma:C0aGradU-L2LinearBound}. Thus, by iterating \cref{lemma:C0aGradU-L2LinearBound} (which is possible thanks to \eqref{eqn:furtherAssC0a}) on a sequence of balls $r_k:=\rho^kr/4$, where we choose $\rho$ such that $\tilde C\rho^{\alpha/2}\le 1$, we obtain a sequence of linear functions $l_k$ such that $l_0\equiv 0$ and 
\begin{align*} 
\frac1{r_k}\left\| u - u(y) - l_k \right\|_{L^{\infty}(B_{r_k}(y))}& \le \rho^{k\alpha/2}\frac1{r/4}\left\| u - u(y)\right\|_{L^{\infty}(B_{r/4}(y))}\\
&\le \rho^{k\alpha/2} \frac{C}{r^{\sfrac{(d+2)}{2}}}\left\| u - u(y) \right\|_{L^{2}(B_{r/2}(y))}.
\end{align*}
Moreover, 
$$|\nabla \ell_{k+1}-\nabla \ell_k|\le \widetilde C\rho^{k\alpha/2} \frac{C}{r^{\sfrac{(d+2)}{2}}}\left\| u - u(y) \right\|_{L^{2}(B_{r/2}(y))},$$
which implies that 
$$|\nabla u(y)|\le\sum_{k=0}^{\infty}|\nabla \ell_{k+1}-\nabla \ell_k|\le \frac{\widetilde C}{1-\rho^{\alpha/2}} \frac{C}{r^{\sfrac{(d+2)}{2}}}\left\| u - u(y) \right\|_{L^{2}(B_{r/2}(y))}.$$
Finally, we notice that 
\begin{align*}
\frac{1}{r^{\sfrac{d}{2}}}\left\| u - u(y) \right\|_{L^{2}(B_{r/2}(y))}&\le C_d|u(y)|+\frac{1}{r^{\sfrac{d}{2}}}\left\| u\right\|_{L^{2}(B_{r/2}(y))}\\
&\le C_d\|u\|_{L^{\infty}(B_{r/2}(x))}+\frac{1}{r^{\sfrac{d}{2}}}\left\| u\right\|_{L^{2}(B_{r/2}(y))}\le \frac{C}{r^{\sfrac{d}{2}}}\left\| u\right\|_{L^{2}(B_{r}(x))},
\end{align*}
where in the last inequality we used the linear bound for $u$ from \eqref{e:linear-bounds}. Thus, we get 
$$|\nabla u(y)|\le \frac{C}{r^{\sfrac{(d+2)}{2}}}\left\| u \right\|_{L^{2}(B_{r}(x))},$$
which implies the second bound in \eqref{e:linear-bounds}.
\end{proof}

\section{Outer and inner variations}\label{section:OuterAndInnerVariations}

In this section we compute the outer and inner variations of the functional \eqref{eqn:quasilinearFunctionalF}, centered at the local minimizer $u$. We use a strategy analogous to  \cite{FerreriSpolaorVelichkov2024:BoundaryBranchingOnePhaseBernoulli, DeLellisSpadaro2016:AreaMinimizingCurrents3BlowUp}. Let the function $\varphi:\R^+ \to \R^+$ be defined as
\begin{equation}\label{eqn:QNCutoffDef}
\varphi(x) \coloneqq
\begin{cases}
    1 & \text{if } x \in (0, 1-\upsilon], \\
    \frac{1-x}{1-\upsilon}  & \text{if } x \in (1-\upsilon, 1], \\
    0 & \text{if } x \in (1, +\infty),
\end{cases}
\end{equation}
for some $\upsilon \in (\sfrac{1}{2}, 1)$ and define the rescaled cutoff functions
\begin{equation}\label{eqn:psiCutoffBr}
\psi_r:\R^d\to\R^+\ ,\qquad \psi_r(x) \coloneqq \varphi\left( \frac{\vert x \vert}{r} \right).
\end{equation}
We introduce the height function,
\begin{equation}\label{eqn:quasilinearFrequencyHeightDef}
    H(r) \coloneqq - \int \varphi'\left( \frac{\vert x \vert }{r} \right) \frac{u(x)^2}{\vert x \vert}
\end{equation}
the energy terms $D_0(r)$, $D_l(r)$ defined as
\begin{align}
    D_0(r) \coloneqq & \int \varphi\left( \frac{\vert x \vert }{r} \right) \left\vert \nabla u \right\vert^2, \label{eqn:quasilinearEnergyD0Def} \\
    D_l(r) \coloneqq &  \int \varphi\left( \frac{\vert x \vert }{r} \right) u\,\partial_sF(x,u,\nabla u) , \label{eqn:quasilinearEnergyDlDef}
\end{align}
and the energy function
\begin{equation}\label{eqn:quasilinearEnergyDef}
    D(r) \coloneqq D_0(r)+D_l(r).
\end{equation}
For notational convenience, let us also define the quantities
\begin{align}
    G(r) \coloneqq & \int \varphi\left( \frac{\vert x \vert }{r} \right) u^2, \label{eqn:quasilinearFrequencyL2Def}\\
    A(r) \coloneqq & - \int \varphi'\left( \frac{\vert x \vert }{r} \right) \vert x \vert \left( \nabla u \cdot \frac{x}{\vert x \vert}\right)^2, \label{eqn:quasilinearFrequencyADef}\\
    B(r) \coloneqq & - \int \varphi'\left( \frac{\vert x \vert }{r} \right) u \, \left(\nabla u \cdot \frac{x}{|x|}\right)\,.\label{eqn:quasilinearFrequencyBDef}
\end{align}
The main content of this section is the following
\begin{lemma}[]\label{lem:quasilinearFreqIdentities}
For all $r \in (0, 1)$, 
we have
\begin{equation}\label{eqn:QN-HDerivative}
H'(r) - \frac{d-1}{r} H(r) - \frac{2}{r} B(r) = 0\,,
\end{equation}
\begin{equation}\label{eqn:QN-OuterVariationD}
D(r) - \frac{1}{r} B(r)+ e_O(r)=0\,,
\end{equation}
\begin{equation}\label{eqn:QN-DDerivative}
(d-2) D(r) - r D'(r) + \frac{2}{r} A(r)
+ e_{I}(r)  = 0\,,
\end{equation}
where the errors $e_O, e_I$ (produced by the outer and the inner variations) are defined as
\begin{align}
e_O(r)& \coloneqq \sum_{k=1}^4 E^{o,k}(r),\label{e:quasilinearEO}\\
e_I(r)& \coloneqq 2\sum_{k=1}^4 E^{i, k}(r).\label{e:quasilinearEI}
\end{align}
and for some constant $C=C(d, \delta_0)$ the following estimates hold, with $\kappa = \gamma/2$:
\begin{align}
|E^{o,1}(r)|&\le C\int \varphi\left( \frac{\vert x \vert}{r}\right) |u|^{2+\kappa}, \label{eq:QNE{o,1}}\\
|E^{o,2}(r)|&\le C\int \varphi\left( \frac{\vert x \vert}{r}\right)|\nabla u|^{2+\kappa}, \label{eq:QNE{o,2}}\\
|E^{o,3}(r)|&\le - C\int \varphi'\left( \frac{\vert x \vert}{r}\right) |u|^{2+\kappa}, \label{eq:QNE{o,3}}\\
|E^{o,4}(r)|&\le - C\int \varphi'\left( \frac{\vert x \vert}{r}\right)|u| |\nabla u|^{1+\kappa}. \label{eq:QNE{o,4}}
\end{align}
\begin{align}
|E^{i,1}(r)|&\le C\int \varphi\left( \frac{\vert x \vert}{r}\right) |u|^2, \label{eq:QNE{i,1}}\\
|E^{i,2}(r)|&\le C\int \varphi\left( \frac{\vert x \vert}{r}\right)|\nabla u|^{2+\kappa}, \label{eq:QNE{i,2}}\\
|E^{i,3}(r)|&\le - C\int \varphi'\left( \frac{\vert x \vert}{r}\right) |u|^2, \label{eq:QNE{i,3}}\\
|E^{i,4}(r)|&\le - C\int \varphi'\left( \frac{\vert x \vert}{r}\right)|\nabla u|^{2+\kappa}. \label{eq:QNE{i,4}}
\end{align}
\end{lemma}

\subsection*{Height derivative: proof of  \eqref{eqn:QN-HDerivative}}
 We can rewrite
    \[
     H(r) =\int_{\R^d} r\varphi\left( \frac{\vert x \vert }{r} \right) \diverg\left( \frac{x}{|x|} \frac{u^2(x)}{\vert x \vert}\right),
    \]
    so that a direct computation gives
    \begin{align*}
    H'(r)
    &=\frac{H(r)}{r}-\frac{2}{r} \int\varphi'\left( \frac{\vert x \vert }{r} \right) u\nabla u\cdot \frac{x}{|x|}-\frac{1}{r} \int |x|\varphi'\left( \frac{\vert x \vert }{r} \right) u^2(x) \diverg\left(\frac{x}{|x|^2}\right) = \\
    &= \frac{d-1}{r} H(r) - \frac{2}{r} \int\varphi'\left( \frac{\vert x \vert }{r} \right) u\left(\nabla u\cdot \frac{x}{|x|}\right)\\
     &= \frac{d-1}{r} H(r) + \frac{2}{r} B(r),
    \end{align*}
    where in the last identity we have used that
    \[
    \diverg\left(\frac{x}{|x|^2}\right)=\frac{d-2}{|x|^2}.
    \]
This concludes the proof of \eqref{eqn:QN-HDerivative}.\qed

\subsection*{Outer variation: proof of \eqref{eqn:QN-OuterVariationD}}

By a direct computation
\begin{align*}
    0&=\frac{d}{dt}\Big|_{t=0}\mathcal F(u+t\,\psi_r\,u;r)\\
    &=\frac{d}{dt}\Big|_{t=0}\int\left(\frac12|\nabla(u+t\,\psi_r\,u)|^2+F(x,u+t\,\psi_r\,u,\nabla(u+t\,\psi_r\,u))\right)\\
     &=\int\Big(\nabla u\cdot\nabla (\psi_r\,u)+\psi_ru\,\partial_sF(x,u,\nabla u)+\nabla(\psi_r\,u)\cdot\nabla_pF(x,u,\nabla u)\Big)\\
      &=\int\Big(\nabla u\cdot\nabla (\psi_r\,u)+\psi_ru\,\partial_sF(x,u,\nabla u)\\
      &\qquad +\psi_r\nabla u\cdot\nabla_pF(x,u,\nabla u)+\frac{u}{r}\phi'\left(\frac{|x|}{r}\right)\frac{x}{|x|}\cdot \nabla_pF(x,u,\nabla u)\Big)\\
    &=\int \varphi\left( \frac{\vert x \vert}{r} \right) \left\vert \nabla u \right\vert^2 + \int \varphi\left( \frac{\vert x \vert }{r} \right) u\,\partial_sF(x,u,\nabla u) + \frac{1}{r} \int \varphi'\left( \frac{\vert x \vert}{r} \right) u \nabla u \cdot \frac{x}{\vert x \vert} + \sum_{k=1}^2 e^{o, k}(r)\\
    &=D(r)-\frac1rB(r)+ \sum_{k=1}^2e^{o, k}(r)\,.
\end{align*}
where we have introduced the errors
\begin{align*}
    e^{o, 1}(r) & \coloneqq  \int \varphi\left(\frac{|x|}{r}\right) \nabla  u\cdot\nabla_pF(x,u,\nabla u), \\
    e^{o, 2}(r) & \coloneqq  \frac{1}{r}\int \varphi'\left(\frac{|x|}{r}\right) u  \nabla_p F(x,u,\nabla u) \cdot \frac{x}{|x|}.
\end{align*}
Now, using the bounds \eqref{e:stima-grad-F-in-X}, \eqref{e:stima-grad-F-in-p} and \eqref{e:stima-derivata-F-in-s} we have the estimates
\begin{align*}
\Big| \nabla  u\cdot\nabla_pF(x,u,\nabla u) \Big|
&\le C(|\nabla u| |u|^{1+\gamma}+|\nabla u|^{2+\gamma}) \\
&\le C\left(|u|^{2+\gamma/2}+|\nabla u|^{2+\gamma/2}\right),
\end{align*}
\begin{align*}
\Big| u \nabla_pF(x,u,\nabla u) \Big|
&\le C(|u| |\nabla u|^{1+\gamma}+|u|^{2+\gamma}).
\end{align*}
Hence, there exist error terms $E^{o, k}(r)$ with $k=1, ..., 4$ with estimates
\begin{align*}
|E^{o,1}(r)|&\le C\int \varphi\left( \frac{\vert x \vert}{r}\right) |u|^{2+\gamma/2}, \\
|E^{o,2}(r)|&\le C\int \varphi\left( \frac{\vert x \vert}{r}\right)|\nabla u|^{2+\gamma/2}, \\
|E^{o,3}(r)|&\le - C\int \varphi'\left( \frac{\vert x \vert}{r}\right) |u|^{2+\gamma}, \\
|E^{o,4}(r)|&\le - C\int \varphi'\left( \frac{\vert x \vert}{r}\right)|u| |\nabla u|^{1+\gamma}.
\end{align*}
such that
\[
D(r)-\frac1rB(r)+ \sum_{k=1}^4 E^{o, k}(r) = 0.
\]
The above computations give the proof of \eqref{eqn:QN-OuterVariationD}, \eqref{eq:QNE{o,1}}, \eqref{eq:QNE{o,2}}, \eqref{eq:QNE{o,3}} and \eqref{eq:QNE{o,4}}.\qed

\subsection*{Inner variation: proof of \eqref{eqn:QN-DDerivative}}
Let $T_{\varepsilon}:B_r \to B_r$ be the family of diffeomorphisms defined as
\[
T_{\varepsilon}(x) \coloneqq x + \varepsilon \psi_r(x) x,
\]
with $\psi_r$ as in \eqref{eqn:psiCutoffBr}. Let $u_{\varepsilon}$ be defined as
\[
u_{\varepsilon} \coloneqq u \circ T_{\varepsilon}^{-1}.
\]
By changing coordinates and differentiating in $\eps$, we get:
\begin{align*}
    \begin{split}
        \int \left(\frac12|\nabla u_\eps|^2+F\left( x, u_{\varepsilon}, \nabla u_{\varepsilon} \right)\right)
        &= \int \left(\frac12|D(T_{\varepsilon}^{-1})[\nabla u \circ T_{\varepsilon}^{-1}]|^2+F\left( x, u \circ T_{\varepsilon}^{-1}, D(T_{\varepsilon}^{-1})[\nabla u \circ T_{\varepsilon}^{-1}]\right) \right)\\
        &= \int \left(\frac12|(DT_{\varepsilon})^{-1}\nabla u|^2+F\left( T_{\varepsilon}, u, (DT_{\varepsilon})^{-1} \nabla u \right)\right) |\det(DT_\varepsilon)| \\
        &= \int \left(\frac12|\nabla u|^2+F(x,u,\nabla u)\right) \\
        &\qquad+\eps\int \Big(\frac12|\nabla u|^2{\rm div}(x\psi_r)-\nabla u\cdot D(\psi_r\,x)\nabla u\Big)\\
        &\qquad\qquad +\eps\int F(x,u,\nabla u)\,{\rm div}(x\psi_r)\\
        & \qquad\qquad\qquad + \eps\int  \psi_r x\cdot \nabla_x F\left( x, u, \nabla u \right)  \\
        & \qquad\qquad\qquad \qquad-\eps\int\nabla_p F\left( x, u, \nabla u \right) \cdot D(\psi_r x) \nabla u  +o(\eps).
    \end{split}
\end{align*}
Now, a standard computation gives
\begin{align*}
\int \Big(\frac12|\nabla u|^2{\rm div}(x\psi_r)-\nabla u\cdot D(\psi_r\,x)\nabla u\Big)&=\frac{d-2}{2} \int \varphi\left( \frac{\vert x \vert}{r} \right) |\nabla u|^2 + \frac{1}{2r} \int \varphi'\left( \frac{\vert x \vert}{r} \right) \vert x \vert |\nabla u|^2  \\
    &\qquad  - \frac{1}{r} \int \varphi'\left( \frac{\vert x \vert}{r} \right) \vert x \vert \left( \nabla u \cdot \frac{x}{\vert x \vert} \right)^2\\   
    &=\frac{d-2}{2} D_0(r) -\frac{r}{2}D_0'(r) +\frac1rA(r)\\
    &=\frac{d-2}{2} D(r) -\frac{r}{2}D'(r) +\frac1rA(r)-\frac{d-2}{2} D_l(r) +\frac{r}{2}D_l'(r).
    \end{align*}
The above computations imply 
\begin{align*}
0=\frac{d}{d\varepsilon}\Big|_{\varepsilon=0}\mathcal F(u_{\eps};r) = \frac{d-2}{2} D(r) -\frac{r}{2}D'(r) +\frac1rA(r) + e^{i,1}(r)+e^{i,2}(r)+\frac{r}{2}D_l'(r),
\end{align*}
where we have defined the errors $e^{i,1}(r)$ and $e^{i,2}(r)$ as
\begin{align*}
    e^{i,1}(r) & \coloneqq \int\psi_r \left[ d \, F(x,u,\nabla u) +   x\cdot \nabla_x F\left( x, u, \nabla u \right)-\nabla u\cdot \nabla_p F\left( x, u, \nabla u \right)\right] + \frac{d-2}{2} D_l(r), \\
    e^{i,2}(r) & \coloneqq \frac{1}{r} \int|x|\varphi'\left(\frac{\vert x \vert}{r}\right) \left[  F(x,u,\nabla u)-\left(\frac{x}{|x|}\cdot\nabla u\right) \left(\frac{x}{|x|}\cdot\nabla_pF(x,u,\nabla u)\right)\right].
\end{align*}
Using the bounds \eqref{e:stima-grad-F-in-X}, \eqref{e:stima-grad-F-in-p} and \eqref{e:stima-derivata-F-in-s}, we can deduce the following estimates:
\begin{align*}
\Big|d \, F(x,u,\nabla u)& +   x\cdot \nabla_x F\left( x, u, \nabla u \right)-\nabla u\cdot \nabla_p F\left( x, u, \nabla u \right)+\frac{d-2}{2}u\partial_sF(x,u,\nabla u)\Big|\\
&\le dC(|u|^{2}+|\nabla u|^{2+\gamma})+|x|C(|u|^{2}+|\nabla u|^{2+\gamma})\\
&\qquad +|\nabla u|C(|u|^{1+\gamma}+|\nabla u|^{1+\gamma})+C(|\nabla u|^{2+\gamma}+|u|^2)\\
&\le C\left(|u|^2+|\nabla u|^{2+\gamma/2}\right),
\end{align*}
\begin{align*}
\Big|F(x,u,\nabla u)&-\left(\frac{x}{|x|}\cdot\nabla u\right) \left(\frac{x}{|x|}\cdot\nabla_pF(x,u,\nabla u)\right)\Big|\\
&\le \left|F(x,u,\nabla u)\right|+\left|\nabla u\right| \left|\nabla_pF(x,u,\nabla u)\right|\\
&\le C\left(|u|^{2}+|\nabla u|^{2+\gamma}\right)+|\nabla u|C\left(|u|^{1+\gamma}+|\nabla u|^{1+\gamma}\right)\\
&\le C\left(|u|^2+|\nabla u|^{2+\gamma/2}\right).
\end{align*}
Finally, we compute the derivative of $D_l$ as
$$D_l'(r) = - \frac{1}{2r} \int|x|\varphi'\left(\frac{\vert x \vert}{r}\right) s \, \partial_s F(x, u, \nabla u),$$
and we notice that 
$$|D_l'(r)| \le  C\int|\varphi'|\left(\frac{\vert x \vert}{r}\right) |u|^2.$$
Combining the above estimates, we get that there exist error terms $E^{i, k}(r)$ with $k=1, ..., 4$ satisfying
\begin{align*}
|E^{i,1}(r)|&\le C\int \varphi\left( \frac{\vert x \vert}{r}\right) |u|^2, \\
|E^{i,2}(r)|&\le C\int \varphi\left( \frac{\vert x \vert}{r}\right)|\nabla u|^{2+\gamma/2}, \\
|E^{i,3}(r)|&\le - C\int \varphi'\left( \frac{\vert x \vert}{r}\right) |u|^2, \\
|E^{i,4}(r)|&\le - C\int \varphi'\left( \frac{\vert x \vert}{r}\right)|\nabla u|^{2+\gamma/2}.
\end{align*}
such that
\[
\frac{d-2}{2} D(r) -\frac{r}{2}D'(r) +\frac1rA(r) + \sum_{k=1}^4 E^{i,k}(r) = 0.
\]
This concludes the proof of \eqref{eqn:QN-DDerivative} and \eqref{eq:QNE{i,1}}, \eqref{eq:QNE{i,2}}, \eqref{eq:QNE{i,3}} and \eqref{eq:QNE{i,4}}.\qed

\section{Scale-preserving $L^{\infty}$-$L^2$ estimates}

The main content of this section is to prove $L^{\infty}-L^2$ estimates for minimizers $u$ of \eqref{eqn:quasilinearFunctionalF-general} at the same scale (see \cref{lemma:QNenergyEstCubesHD} below). To this aim we use the strategy of \cite{FerreriSpolaorVelichkov2024:BoundaryBranchingOnePhaseBernoulli}, thus we proceed via a Whitney decomposition argument in the spirit of Almgren-De Lellis-Spadaro \cite{Almgren2000:AlmgrensBigRegularityPaper, DeLellisSpadaro2016:AreaMinimizingCurrents1LpEstimates, DeLellisSpadaro2016:AreaMinimizingCurrents2CenterManifold, DeLellisSpadaro2016:AreaMinimizingCurrents3BlowUp} (see \cref{section:QNWhitneyDecomp} below).

In the following of this section, it will be useful to work under the following
\begin{assumptions}\label{assumptions:QNvanishingAssumptions}
    The origin is a singular point in the sense that
    \[
    u(0) = 0 \quad \text{and} \quad \nabla u(0) = 0.
    \]
\end{assumptions}

\subsection{Some weighted inequalities}
In this section we mainly recall some weighted inequalities that were derived in \cite{FerreriSpolaorVelichkov2024:BoundaryBranchingOnePhaseBernoulli}, which will be useful for the subsequent analysis.

\subsubsection*{Height inequality}

To begin with, we recall a bound from \cite{FerreriSpolaorVelichkov2024:BoundaryBranchingOnePhaseBernoulli} for the weighted $L^2$ norm $G(r)$ from \eqref{eqn:quasilinearFrequencyL2Def} in terms of the height function $H(r)$ from \eqref{eqn:quasilinearFrequencyHeightDef}.

\begin{lemma}[{\cite[Lemma 5.2]{FerreriSpolaorVelichkov2024:BoundaryBranchingOnePhaseBernoulli}}]\label{lemma:QNheight-function-integral-inequality}
For every function $u\in H^1(B_r)$
\begin{gather}
    G(r)= \int_{B_r} \varphi\left(\frac{|x|}{r}\right)\,u^2\le \int_0^r\,H(\rho)\,d\rho,
\end{gather}
where $H$ is the height function from \eqref{eqn:quasilinearFrequencyHeightDef}.
\end{lemma}

\subsubsection*{Poincar\'e inequality}

Now recall a weighted Poincar\'e-type inequality from \cite{FerreriSpolaorVelichkov2024:BoundaryBranchingOnePhaseBernoulli}.
\begin{lemma}[{\cite[Lemma 5.4]{FerreriSpolaorVelichkov2024:BoundaryBranchingOnePhaseBernoulli}}]\label{lemma:QNweightedPoincare}
    There exist constants $c=c(d, \delta_0)>0$ and $r_0=r_0(d, \delta_0)>0$ with the following property. Suppose that the function $u$ is a local minimizer of \eqref{eqn:quasilinearFunctionalF-general} under \cref{assumptions:QNuC1aAprioriEst} and \cref{assumptions:QNvanishingAssumptions}, then
    \begin{equation}\label{eqn:weightedPoincare}
    r\,D_0(r)\geq c\, H(r) \text{ for every } r \in (0, r_0).
    \end{equation}
\end{lemma}
As a consequence of the $L^2$-norm bound \cref{lemma:QNheight-function-integral-inequality}, the Poincar\'e inequality from \cref{lemma:QNweightedPoincare} and the bound \eqref{e:stima-derivata-F-in-s}, we have the following
\begin{corollary}\label{corollary:QNRHSEnergyControl}
There exist constants $c=c(d, \delta_0)>0$ and $r_0=r_0(d, \delta_0)>0$ with the following property. Suppose that the function $u$ is a local minimizer of \eqref{eqn:quasilinearFunctionalF-general} under \cref{assumptions:QNuC1aAprioriEst} and \cref{assumptions:QNvanishingAssumptions}, then
\begin{itemize}
    \item[(i)]
    \[
       G(r) \le c\, r^2 \, D_0(r)\quad \text{ for every }\quad r \in (0, r_0),
    \]

    \item[(ii)]
    \[
       0 < (1-cr^{\alpha\gamma}) D_0(r)\leq D(r)\leq (1+cr^{\alpha\gamma}) D_0(r)\, \text{ for every } r \in (0, r_0).
    \]
\end{itemize}
\end{corollary}
\begin{proof}
Point (i) follows as in \cite[Corollary 5.5]{FerreriSpolaorVelichkov2024:BoundaryBranchingOnePhaseBernoulli}. Concerning point (ii), from \eqref{e:stima-derivata-F-in-s}, \eqref{eqn:quasilinearEnergyDlDef} and point (i) we see that
\begin{align*}
\left\vert D_l(r) \right\vert & \le C \int \varphi\left( \frac{\vert x \vert }{r} \right) \left(\vert u \vert^2 + \vert \nabla u \vert^{2+\gamma}\right) \\
& \le C \left[ G(r) + r^{\alpha\gamma} D_0(r) \right] \\
& \le C \left[ r D_0(r) + r^{\alpha\gamma} D_0(r) \right],
\end{align*}
so that the conclusion follows.
\end{proof}

\subsection{Whitney decomposition}\label{section:QNWhitneyDecomp}
In the proof of the monotoncity of the frequency function (\cref{thm:QNfreqmon}), the estimates of the error terms produced from the nonlinearity $F$ rely on a Whitney decomposition type argument (see \cref{sub:error-estimates}). The construction of this Whitney decomposition is exactly the same as in \cite{FerreriSpolaorVelichkov2024:BoundaryBranchingOnePhaseBernoulli}, but since it is a key step in the proof of the main theorems (\cref{theorem:quasilinearUniqueContinuation} and \cref{thm:QNfreqmon}), we explain the detailed construction in this subsection, keeping the notations from \cite{FerreriSpolaorVelichkov2024:BoundaryBranchingOnePhaseBernoulli}.\medskip

Consider a function $u:B_R\to\R$, $u\in H^1(B_R)$, defined in some sufficiently large ball $B_R\subset\R^d$ with $R$ chosen in such a way that the cube $[-4,4]^d$ is contained in $B_R$. We define the Whitney decomposition of the cube $[-1,1]^d$ as follows.\medskip 

\noindent {\bf Basic notations.} Given $a=(a_1,\dots,a_d)\in\R^d$ and $\ell>0$, we denote by $L=L_\ell(a)$ the closed cube of center $a$ and side $2\ell$ as follows 
\begin{equation}\label{e:definition-of-a-cube}
L:=[a_1-l, \, a_1+l]\times  \dots\times [a_d-l, \, a_d+l]\,.
\end{equation}
Vice versa, for a cube $L$ of the form \eqref{e:definition-of-a-cube}, we will use the notation 
$$a(L):=(a_1,\dots,a_d)\qquad\text{and}\qquad \ell(L):=\ell.$$
Moreover, we will denote by $B_L$ the ball 
$$B_L:=B_{3 l(L)}\left( a(L) \right).$$ 
\noindent {\bf Collections of cubes.} We define the collections of cubes $\mathcal C_j$, $j\ge 1$, as follows: the only element of the set $\mathcal C_1$ is the cube $[-1,1]^{d}$; the collection of cubes $\mathcal C_2$ is obtained by dividing $[-1,1]^{d}$ into $3^d$ cubes with disjoint interiors and with the same side-length; similarly, for every $j\ge 1$, the collection $\mathcal C_{j+1}$ is obtained by dividing each of the cubes from $\mathcal C_{j}$ into $3^d$ cubes with disjoint interiors and with the same side-length. In particular, if $L\in \mathcal C_j$, for some $j$, then 
$$l(L)=3^{1-j}\qquad\text{and}\qquad a(L)\in (3^{1-j} \mathbb Z)^d.$$
By construction, if $L\in \mathcal C_j$ and $H\in \mathcal C_k$ for some $k>j$, then we have only two possibilities:
\begin{enumerate}[(1)]
\item $H\subset L$;
\item $L$ and $H$ have disjoint interiors. 
\end{enumerate}
If case (1) occurs, then we say that $H$ is a {\it descendant} of $L$ and that $L$ is an {\it ancestor} of $H$.
Moreover, if two cubes $H$ and $L$ are such that $L\in \mathcal C_j$, $H\in \mathcal C_{j+1}$ and $H\subset L$, we will say that {\it $L$ is the father of $H$} and that {\it $H$ is a son of $L$}. \medskip

\noindent {\bf Whitney decomposition.} 
From now on we fix two constants 
\begin{equation}\label{e:QN-constants-whitney}
C_0 > 0\qquad\text{and}\qquad \alpha \in (0, 1/2).
\end{equation}
We define the family of cubes (with disjoint interiors) $\mathcal W$ as: 
$$\mathcal W= \mathcal W^e\cup \mathcal W^h,$$ 
where of the family of {\it excess cubes} $\mathcal W^e$ and the family of {\it height cubes} $\mathcal W^h$, are the unions
$$\mathcal W^{e}=\bigcup_j\mathcal W^{e}_j\qquad\text{and}\qquad\mathcal W^{h}=\bigcup_j\mathcal W^{h}_j.$$
We construct the families of cubes $\mathcal W^{h}_j$ and $\mathcal W^{e}_j$ inductively. When $j=0$, we set $\mathcal W_0=\emptyset$. For $j\ge 1$, the families $\mathcal W^{h}_j$ and $\mathcal W^{e}_j$ are disjoint subsets of the collection $\mathcal C_j$ and are obtained as follows. Consider a cube $L\in \mathcal C_j$ such that 
\begin{center}
no ancestor of $L$ is in $\displaystyle\bigcup_{i=1}^{j-1}\mathcal W^{e}_i$ or in $\displaystyle\bigcup_{i=1}^{j-1}\mathcal W^{h}_i$.
\end{center}
\begin{enumerate}[(1)]
    \item We say that $L\in \mathcal W^e_j$ if
\begin{equation}\label{eqn:QNexcessDecayPropertyDefGrid}
\int_{B_L} \vert \nabla u \vert^2 \geq C_0 \,l(L)^{d+2\alpha}\,, 
\end{equation}  
where $C_0$ and $\alpha$ are the constants from \eqref{e:QN-constants-whitney}.
\item We say that $L\in \mathcal W^h_j$, if $L\notin \mathcal W^e_j$ and 
\begin{equation}\label{eqn:QNheightDecayPropertyDefGrid}
\int_{B_L}  u^2 \geq C_0 \, l(L)^{d+2+2\alpha} \,,
\end{equation}
where again $C_0$ and $\alpha$ are the constants from \eqref{e:QN-constants-whitney}.
\item If none of the above occurs we say that $L\in \mathcal S_j$.
\end{enumerate}
It is immediate to check that the decomposition $\mathcal W$ has the following properties:
\begin{itemize}
\item for every $j$, $\mathcal W^h_j$, $\mathcal W^e_j$ and $\mathcal S_j$ are disjoint subsets of $\mathcal C_j$;
\item $\mathcal W$ is a countable union of cubes with disjoint interior; 
\item the residual set of points $\displaystyle[-1,1]^d\setminus \bigcup_{L\in \mathcal W}L$ is contained in the compact set
\begin{equation}\label{e:QN-residual-set-whitney}
\Gamma:=\bigcap_{j\ge 1}\bigcup_{L\in \mathcal S_j}L\,;
\end{equation}
\item for every $x_0\in \Gamma$ it holds
\begin{equation}\label{eq:QN-contactset}
u(x_0)=0\qquad \text{and}\qquad \nabla u(x_0)=0\,;   
\end{equation}
\item if $L\in \mathcal W^e$ and $H$ is the father of $L$, then $H\notin  \mathcal W^e$, and $H\notin \mathcal W^h$, so we have
\begin{equation}\label{eqn:excessQNL2SameScale}
\int_{B_H} u^2 \le C \, l(L)^2 \int_{B_L} \vert \nabla u \vert^2 \quad \text{and} \quad \int_{B_H} \vert \nabla u \vert^2 \le C \int_{B_L} \vert \nabla u \vert^2,
\end{equation}
where $C$ depends only on the dimension $d$  and the constants $C_0$ and $\alpha$ from \eqref{e:QN-constants-whitney};
\item finally, if $L\in \mathcal W^h$ and $H$ is the father of $L$, then $L\notin \mathcal W^e$, $H\notin  \mathcal W^e$, $H\notin \mathcal W^h$, and
\begin{equation}\label{eqn:heightQNL2SameScale}
\int_{B_H} u^2 \le C \int_{B_L} u^2 \quad \text{and} \quad \int_{B_H} \vert \nabla u \vert^2 \le \frac{C}{l(L)^2} \int_{B_L} u^2 ,
\end{equation}
where as above $C$ depends on $C_0$, $\alpha$, and $d$.
\end{itemize}
We conclude this section with the following lemma, which contains two properties of the Whitney decomposition for solutions $u$ satisfying \cref{assumptions:QNuC1aAprioriEst} and \cref{assumptions:QNvanishingAssumptions}. For the proof we refer to {\cite[Lemma 5.8]{FerreriSpolaorVelichkov2024:BoundaryBranchingOnePhaseBernoulli}. 
\begin{lemma}[{\cite[Lemma 5.8]{FerreriSpolaorVelichkov2024:BoundaryBranchingOnePhaseBernoulli}}]\label{lemma:QNenergyEstCubesHD}
There exist constants $R=R(d, \delta_0)>0$, $\lambda=\lambda(d, \delta_0)>0$ and $C = C(d, \delta_0)>0$ with the following property. Suppose that the function $u$ is a minimizer of \eqref{eqn:quasilinearFunctionalF-general} under \cref{assumptions:QNuC1aAprioriEst} and \cref{assumptions:QNvanishingAssumptions}. Then, for all cubes $L \in \mathcal{W}$ with
\[
L \cap B_r \neq \emptyset,
\]
the following estimate holds:
    \[
    \| u \|_{L^{\infty}\left( L \right)}+\| \nabla u \|_{L^{\infty}\left( L \right)} \le C D_0(r)^{\lambda} \quad \text{for all } r \in (0, R).
    \]
\end{lemma}

\section{Frequency (almost-)monotonicity}
Let us introduce the Almgren-type frequency function $N(r)$ defined as
\begin{equation}\label{eqn:quasilinearFrequencyN(r)Def}
    N(r) \coloneqq \frac{r D(r)}{H(r)}. 
\end{equation}
The main content of this section is the following
\begin{theorem}[]\label{thm:QNfreqmon}
   There exist constants $R=R(d, \delta_0)>0$, $\lambda=\lambda(d, \delta_0)>0$ and $C = C(d, \delta_0)>0$ with the following property. Suppose that the function $u$ in $B_R$ is a local minimizer of \eqref{eqn:quasilinearFunctionalF-general} under \cref{assumptions:QNuC1aAprioriEst} and \cref{assumptions:QNvanishingAssumptions}, and that
   \begin{equation*}
   H(r_0)>0 \quad \text{for some } r_0 \in (0, R).
   \end{equation*}
   Then,
   \begin{equation}\label{eqn:QNFrequencyAlmostMonocotone}
        e^{\,g(r)} N(r) \text{ is non-decreasing in a neighborhood of } r_0,
    \end{equation}
    where the function $g(r): \R^+ \to \R$ is defined as
    \begin{equation}\label{eqn:QNFrequencyg(r)Def}
    g(r) \coloneqq \frac{C}{\beta} \left[ r^{\beta} + D(r)^{\beta} \right]
    \end{equation}
    and satisfies
    \begin{equation}\label{eqn:QNFrequencyg(r)Limit0}
    g(r) \to 0 \text{ as } r \to 0^+.
    \end{equation}
\end{theorem}

\subsection{Frequency derivative}
To begin with, we compute the derivative of the frequency function intoduced in \eqref{eqn:quasilinearFrequencyN(r)Def}. To this aim, we first introduce the auxiliary quantity
\begin{equation}\label{eqn:quasilinearF(r)Def}
    F(r) \coloneqq \frac{1}{r} B(r) - E^{o,4}(r),
\end{equation}
and then we prove the following
\begin{lemma}\label{lemma:QNLogFrequencyDerivative}
 There exists a constant $R=R(d, \delta_0)>0$ with the following property. Suppose that
\[
H(r_0) > 0 \quad\text{for some}\quad r_0 \in (0, R).
\]
Then, for all $r$ in a neighborhood of $r_0$, the following identity holds
\begin{equation}\label{eqn:QNLogFrequencyDerivative}
    \frac{d}{dr} \ln{N(r)} = \frac{1}{r} + \frac{D'}{D} - \frac{H'}{H} = \frac{2}{r^2} \frac{1}{F(r) H(r)} \left[A(r)H(r) - B(r)^2 \right]+ \sum_{k=1}^3 e_k(r),
\end{equation}
where we have defined the error terms in the following way:
\begin{align}
    & e_1(r) \coloneqq  \frac{1}{r} \frac{\sum_{k=1}^4 E^{i, k}(r)}{D(r)}, \label{eqn:QNerre_1Def}\\
    & e_2(r) \coloneqq - \frac{1}{r^2}\frac{A(r)}{D(r)F(r)} \sum_{k=1}^3 E^{o, k}(r), \label{eqn:QNerre_2Def}\\
    & e_3(r) \coloneqq  \frac{1}{r}\frac{B(r) E^{o,4}(r)}{F(r) H(r)}. \label{eqn:QNerre_3Def}
\end{align}
\end{lemma}
\begin{proof}
To begin with, we have that
\begin{equation}\label{eqn:QNlnNFrequencyDerivativeCOmputation}
\frac{d}{dr} \ln{N(r)} = \frac{1}{r} + \frac{D'(r)}{D(r)} - \frac{H'(r)}{H(r)}.
\end{equation}
Now, from \eqref{eqn:QN-DDerivative} we have
\begin{equation}\label{eqn:QND'DFrequencyDerivativeCOmputation}
\frac{D'(r)}{D(r)} = \frac{d-2}{r} +\frac{2}{r^2} \frac{A(r)}{D(r)} + e_3(r).
\end{equation}
On the other hand, \eqref{eqn:QN-HDerivative} implies that
\begin{equation}\label{eqn:QNH'HFrequencyDerivativeCOmputation}
\frac{H'(r)}{H(r)} = \frac{d-1}{r} + \frac{2}{r}\frac{B(r)}{H(r)}.
\end{equation}
Combining \eqref{eqn:QNlnNFrequencyDerivativeCOmputation}, \eqref{eqn:QND'DFrequencyDerivativeCOmputation} and \eqref{eqn:QNH'HFrequencyDerivativeCOmputation} we deduce
\begin{equation}\label{eqn:QNlnNFrequencyDerivativeComputation}
\frac{d}{dr} \ln{N(r)} = 2 \left[\frac{1}{r^2} \frac{A(r)}{D(r)} - \frac{1}{r} \frac{B(r)}{H(r)} \right] + e_3(r).
\end{equation}
Since we wish to avoid comparing most error terms with the height $H(r)$, we can split the terms in square brackets as
\begin{align*}
2 \left[\frac{1}{r^2} \frac{A(r)}{D(r)} - \frac{1}{r} \frac{B(r)}{H(r)} \right] & = 2 \left[\frac{1}{r^2} \frac{A(r)}{F(r)} - \frac{1}{r} \frac{B(r)}{H(r)} \right] + 2 \frac{A(r)}{r^2} \left[\frac{1}{D(r)} - \frac{1}{F(r)} \right] \\
& = 2 \left[\frac{1}{r^2} \frac{A(r)}{F(r)} - \frac{1}{r} \frac{B(r)}{H(r)} \right] + e_2(r),
\end{align*}
where the quantity $F(r)$ is the one from \eqref{eqn:quasilinearF(r)Def}. Hence, we can rewrite \eqref{eqn:QNlnNFrequencyDerivativeCOmputation} as
\begin{align*}
    \frac{d}{dr} \ln{N(r)} & = \frac{2}{r^2} \frac{1}{F(r) H(r)} \left[A(r)H(r) - rF(r)B(r) \right] + e_1(r) + e_2(r) \\
    & = \frac{2}{r^2} \frac{1}{F(r) H(r)} \left[A(r)H(r) - B(r)^2 \right] + \frac{2}{r^2}\frac{B(r)}{F(r) H(r)} \left( B(r) - rF(r) \right) + e_1(r) + e_2(r) \\
    & = \frac{2}{r^2} \frac{1}{F(r) H(r)} \left[A(r)H(r) - B(r)^2 \right] + e_1(r) + e_2(r) + e_3(r) ,
\end{align*}
which is exactly \eqref{eqn:QNLogFrequencyDerivative}.
\end{proof}

\subsection{Error estimates}\label{sub:error-estimates}
The (almost-)monotonicity of the frequency function will follow as a consequence of the following proposition, which deals with the estimate for the error terms.
\begin{proposition}\label{prop:QNFrequncyErrsEst}
    There exists constants $R=R(d, \delta_0)>0$, $C=C(d, \delta_0)$, $\beta=\beta(d, \delta_0)$ with the following properties. Suppose that the function $u$ is a local minimizer of \eqref{eqn:quasilinearFunctionalF-general} under \cref{assumptions:QNuC1aAprioriEst} and \cref{assumptions:QNvanishingAssumptions} and that
    \begin{equation*}
        H(r_0 > 0) \quad \text{for some } r_0 \in (0, R).
    \end{equation*}
    Then, for all $r$ in a neighborhood of $r_0$ the following estimates hold:
    \begin{align}
        & \vert e_1(r) \vert \le C \left[ r^{\beta-1} + D_0(r)^{\beta-1} D_0'(r) \right], \label{eqn:errQNEste1} \\
        & \vert e_2(r) \vert \le C D_0(r)^{\beta-1} D_0'(r), \label{eqn:errQNEste2} \\
        & \vert e_3(r) \vert \le C D_0(r)^{\beta-1} D_0'(r). \label{eqn:errQNEste3}
    \end{align}
\end{proposition}
We proceed to prove the estimates \eqref{eqn:errQNEste1}, \eqref{eqn:errQNEste2} and \eqref{eqn:errQNEste3} in order.

\subsection*{Proof of \eqref{eqn:errQNEste1}}
To begin with, combining the estimates \eqref{eq:QNE{i,1}}, \eqref{eq:QNE{i,2}}, \eqref{eq:QNE{i,3}} and \eqref{eq:QNE{i,4}}, together with the Whitney decomposition from \cref{section:QNWhitneyDecomp}, the bounds from \cref{lemma:QNenergyEstCubesHD}, \cref{lemma:QNheight-function-integral-inequality} and \cref{corollary:QNRHSEnergyControl}, we have
\begin{align}\label{eqn:QNE^{i, 1}(r)Est}
    \begin{split}
        \left\vert E^{i, 1}(r) \right\vert & \le G(r) \le C r^2 D_0(r),
    \end{split}
\end{align}
\begin{align}\label{eqn:QNE^{i, 2}(r)Est}
    \begin{split}
        \left\vert E^{i, 2}(r) \right\vert & \le C \sum_{L \in \mathcal{W}}  \int_{L} \varphi\left( \frac{\vert x \vert}{r}\right)|\nabla u|^{2+\kappa} \\
        & \le C \left\| \nabla u \right\|_{L^{\infty}\left( L \right)}^{\kappa} \int_{B_r} \varphi\left( \frac{\vert x \vert}{r}\right)|\nabla u|^2  \le C D_0(r)^{1+\lambda\kappa},
    \end{split}
\end{align}
\begin{align}\label{eqn:QNE^{i, 3}(r)Est}
    \begin{split}
        \left\vert E^{i, 3}(r) \right\vert & \le C r H(r) \le C  r^2 D_0(r),
    \end{split}
\end{align}
\begin{align}\label{eqn:QNE^{i, 4}(r)Est}
    \begin{split}
        \left\vert E^{i, 4}(r) \right\vert & \le -C \sum_{L \in \mathcal{W}} \int_{L \cap B_r} \varphi'\left( \frac{\vert x \vert}{r}\right) |\nabla u|^{2+\kappa} \\
        & \le -C \left\| \nabla u \right\|_{L^{\infty}\left( L \right)}^{\kappa} \int_{B_r} \varphi'\left( \frac{\vert x \vert}{r}\right)|\nabla u|^2 \le C r D_0(r)^{\lambda\kappa} D_0'(r).
    \end{split}
\end{align}
Combining \eqref{eqn:QNE^{i, 1}(r)Est}, \eqref{eqn:QNE^{i, 2}(r)Est}, \eqref{eqn:QNE^{i, 3}(r)Est}, and \eqref{eqn:QNE^{i, 4}(r)Est} with the definition of the error $e_1(r)$ from \eqref{eqn:QNerre_1Def} gives
\[
\left\vert e_1(r) \right\vert \le C \left[ r^{\alpha\lambda\kappa - 1} + D_0^{1-\lambda\kappa} D_0'(r) \right],
\]
where we have used \cref{assumptions:QNuC1aAprioriEst} and \cref{assumptions:QNvanishingAssumptions}. This concludes the proof of \eqref{eqn:errQNEste1}.\qed

\subsection*{Proof of \eqref{eqn:errQNEste2}}
From the bounds \eqref{eq:QNE{o,1}}, \eqref{eq:QNE{o,2}} and \eqref{eq:QNE{o,3}}, using the Whitney decomposition from \cref{section:QNWhitneyDecomp} together with \cref{lemma:QNenergyEstCubesHD} and \cref{corollary:QNRHSEnergyControl}, we have the estimates
\begin{align}\label{eqn:QNE^{o, 1}(r)Est}
    \begin{split}
        \left\vert E^{o, 1}(r) \right\vert & \le C \sum_{L \in \mathcal{W}}  \int_{L} \varphi\left( \frac{\vert x \vert}{r}\right) |u|^{2+\kappa} \\
        & \le C \left\| u \right\|_{L^{\infty}\left( L \right)}^{\kappa} \int_{B_r} \varphi\left( \frac{\vert x \vert}{r}\right)|u|^2  \le C D_0(r)^{\lambda\kappa} G(r)  \le C r^2 D_0(r)^{1+\lambda\kappa},
    \end{split}
\end{align}
\begin{align}\label{eqn:QNE^{o, 2}(r)Est}
    \begin{split}
        \left\vert E^{o, 2}(r) \right\vert & \le C \sum_{L \in \mathcal{W}}  \int_{L} \varphi\left( \frac{\vert x \vert}{r}\right)|\nabla u|^{2+\kappa} \\
        & \le C \left\| \nabla u \right\|_{L^{\infty}\left( L \right)}^{\kappa} \int_{B_r} \varphi\left( \frac{\vert x \vert}{r}\right)|\nabla u|^2  \le C D_0(r)^{1+\lambda\kappa},
    \end{split}
\end{align}
\begin{align}\label{eqn:QNE^{o, 3}(r)Est}
    \begin{split}
        \left\vert E^{o, 3}(r) \right\vert & \le -C \sum_{L \in \mathcal{W}} \int_{L \cap B_r} \varphi'\left( \frac{\vert x \vert}{r}\right) |u|^{2+\kappa} \\
        & \le -C \left\| u \right\|_{L^{\infty}\left( L \right)}^{\kappa} \int_{B_r} \varphi'\left( \frac{\vert x \vert}{r}\right)|u|^2  \le C r D_0(r)^{\lambda\kappa} H(r)  \le C r^2 D_0(r)^{1+\lambda\kappa},
    \end{split}
\end{align}
where in the last inequality we have also used \cref{lemma:QNheight-function-integral-inequality}. In particular, combining \eqref{eqn:QNE^{o, 1}(r)Est}, \eqref{eqn:QNE^{o, 2}(r)Est} and \eqref{eqn:QNE^{o, 3}(r)Est} we get that
\begin{align}\label{eqn:QNSumE^{o, k}(r)Est}
    \begin{split}
        \sum_{k=1}^3 \left\vert E^{o, 3}(r) \right\vert & \le C D_0(r)^{1+\lambda\kappa}  \le C r^{\alpha\lambda\kappa} D_0(r),
    \end{split}
\end{align}
where in the last passage we have used \cref{assumptions:QNuC1aAprioriEst} and \cref{assumptions:QNvanishingAssumptions}. Combining the bound \eqref{eqn:QNSumE^{o, k}(r)Est} with the definition of the quantity $F(r)$ from \eqref{eqn:quasilinearF(r)Def} and the outer variation \eqref{eqn:QN-OuterVariationD} gives the equivalence
\begin{align}\label{eqn:QNF(r)EquivD_0(r)Est}
    \begin{split}
        \left( 1- C r^{\alpha\lambda\kappa} \right) D_0(r) \le F(r) \le \left( 1 + C r^{\alpha\lambda\kappa} \right) D_0(r).
    \end{split}
\end{align}
Now, the definition of $A(r)$ from \eqref{eqn:quasilinearFrequencyADef} implies that
\begin{equation}\label{eqn:QNA(r)Est}
        \left\vert A(r) \right\vert \le C r^2 D_0'(r),
\end{equation}
so that, combining the definition of $e_2(r)$ from \eqref{eqn:QNerre_2Def} with the estimates \eqref{eqn:QNSumE^{o, k}(r)Est} and \eqref{eqn:QNA(r)Est} we have
\[
\left\vert e_2(r) \right\vert \le C D_0^{1-\lambda\kappa} D_0'(r),
\]
which is exactly \eqref{eqn:errQNEste2}.\qed

\subsection*{Proof of \eqref{eqn:errQNEste3}}
From the estimate \eqref{eq:QNE{o,4}}, using the Whitney decomposition from \cref{section:QNWhitneyDecomp} together with \cref{lemma:QNenergyEstCubesHD}, we have that
\begin{align}\label{eqn:QNE^{o, 4}(r)Est}
    \begin{split}
        \left\vert E^{o, 4}(r) \right\vert & \le -C \sum_{L \in \mathcal{W}} \int_{L \cap B_r} \varphi'\left( \frac{\vert x \vert}{r}\right)|u| |\nabla u|^{1+\kappa} \\
        & \le -C \left\| \nabla u \right\|_{L^{\infty}\left( L \right)}^{\kappa} \int_{B_r} \varphi'\left( \frac{\vert x \vert}{r}\right)|u| |\nabla u| \\
        & \le C D_0(r)^{\lambda\kappa} \left( r^2 H(r) D_0'(r) \right)^{1/2},
    \end{split}
\end{align}
where in the last inequality we have also used the definitions of $H(r)$ and $D_0(r)$ from \eqref{eqn:quasilinearFrequencyHeightDef} and \eqref{eqn:quasilinearEnergyD0Def} respectively. Moreover, from the definition of $B(r)$ in \eqref{eqn:quasilinearFrequencyBDef}, we also see that
\begin{equation}\label{eqn:QNB(r)Est}
    \left\vert B(r) \right\vert \le C \left( r^2 H(r) D_0'(r) \right)^{1/2}.
\end{equation}
Combining the estimates \eqref{eqn:QNE^{o, 4}(r)Est} and \eqref{eqn:QNB(r)Est} with the definition of the error $e_1(r)$ from \eqref{eqn:QNerre_1Def}, \cref{corollary:QNRHSEnergyControl} (ii) and the equivalence \eqref{eqn:QNF(r)EquivD_0(r)Est}, we have
\begin{equation*}
    \left\vert e_3(r) \right\vert \le C D_0(r)^{1-\lambda\kappa}  D_0'(r),
\end{equation*}
which concludes the proof of \eqref{eqn:errQNEste3}.\qed

\subsection{Proof of \cref{thm:QNfreqmon}}
To begin with, we observe that thanks to the hypothesis $H(r)>0$ the frequency function is well defined. Moreover, combining \cref{lemma:QNLogFrequencyDerivative} and \cref{prop:QNFrequncyErrsEst} we have
\begin{equation}\label{eqn:QNFrequencyDerivativeEst1}
    \frac{d}{dr} \ln{N(r)} \ge \frac{2}{r^2} \frac{1}{F(r) H(r)} \left[A(r)H(r) - B(r)^2 \right] - C \left[ r^{\beta-1} + D_0(r)^{\beta-1} D_0'(r) \right]
\end{equation}
in a neighborhood of $r_0$ and for some constants $C=C(d, \delta_0)>0$ and $\beta=\beta(d, \delta_0)>0$. By a standard Cauchy-Schwarz inequality 
\[
A(r)H(r) - B(r)^2 \ge 0,
\]
so that, from \eqref{eqn:QNF(r)EquivD_0(r)Est} and \eqref{eqn:QNFrequencyDerivativeEst1}
\begin{equation}\label{eqn:QNFrequencyDerivativeEst2}
     \frac{d}{dr} \ln{N(r)} \ge - C \left[ r^{\beta-1} + D_0(r)^{\beta-1} D_0'(r) \right].
\end{equation}
Now let the function $g(r)$ as in \eqref{eqn:QNFrequencyg(r)Def}. Using the estimate \eqref{eqn:QNFrequencyDerivativeEst2} we have
\begin{align*}
    \begin{split}
        \frac{d}{dr} e^{g(r)} N(r) = e^{g(r)} \left[ N'(r) + g'(r) N(r) \right] \ge 0,
    \end{split}
\end{align*}
in a neighborhood of $r_0$, which gives the monotonicity \eqref{thm:QNfreqmon}. To conclude, the condition \eqref{eqn:QNFrequencyg(r)Limit0} follows at once combining the definition of $D_0(r)$ in \eqref{eqn:quasilinearEnergyD0Def} together with \cref{assumptions:QNuC1aAprioriEst} and \cref{assumptions:QNvanishingAssumptions}.\qed

\section{Proof of \cref{theorem:quasilinearUniqueContinuation}}
Without loss of generality we can assume $x_0 = 0$, and let $R=R(d, \delta_0)>0$ be the radius from \cref{thm:QNfreqmon}. In order to prove \cref{theorem:quasilinearUniqueContinuation} it is sufficient to show that
\begin{equation}\label{eqn:uEquiv0BRProof}
    u\equiv 0
\end{equation}
Indeed, if \eqref{eqn:uEquiv0BRProof} holds true, we can can iterate \cref{theorem:quasilinearUniqueContinuation} for any point in $B_R$ and so on, thus covering all $B$.

To begin with, thanks to \eqref{eqn:uEquiv0BRProof} we can suppose that there exists a radius $r_0 \in (0, R)$ such that
\[
H(r_0)>0.
\]
%
Let also
\begin{equation*}
    r_1 \coloneqq \sup \left\{ r \in [0, r_0] \, : \, H(r) = 0 \right\}.
\end{equation*}
Combining the height derivative \eqref{eqn:QN-HDerivative} and the outer variation \eqref{eqn:QN-OuterVariationD} we have that
\[
H'(r) = \frac{d-1}{r} H(r) + 2 D(r) + 2 \sum_{k=1}^4 E^{o,k},
\]
and estimating the errors $E^{o, k}$ using \eqref{eq:QNE{o,1}}, \eqref{eq:QNE{o,2}}, \eqref{eq:QNE{o,3}} and \eqref{eq:QNE{o,4}} 
\begin{align}\label{eqn:QNlndHBoundProof1}
    \begin{split}
        H'(r) & \le \frac{d-1}{r} H(r) + C D_0(r) + D_0(r)^{\beta} \left( r^2 H(r) D_0'(r) \right)^{1/2} \\
        & \le \frac{C}{r} H(r) + r^2 D_0(r)^{\beta} D_0'(r),
    \end{split}
\end{align}
for some constant $C=C(d, \delta_0)$ and all $r \in (r_1, r_0)$, where we have also used \cref{corollary:QNRHSEnergyControl} (ii). Now, from the (almost-)monotonicity of the frequency function \eqref{eqn:QNFrequencyAlmostMonocotone}, we have that
\begin{equation}\label{eqn:QNFrequencyBoundAbove}
    H(r) \ge \widetilde{C} r D_0(r) \quad \text{for all } r \in (r_1, r_0),
\end{equation}
and a constant $\widetilde C \coloneqq C N(r_0)$ where $C=C(d, \delta_0)>0$, thanks to \eqref{eqn:QNFrequencyg(r)Limit0}. In particular, from \eqref{eqn:QNlndHBoundProof1} and \eqref{eqn:QNFrequencyBoundAbove} we see that
\begin{equation}\label{eqn:QNlndHBoundProof2}
    \frac{H'(r)}{H(r)} \le \frac{C}{r} + \widetilde{C} D_0(r)^{\beta-1} D_0'(r) \quad \text{for all } r \in (r_1, r_0).
\end{equation}
Integrating the estimate \eqref{eqn:QNlndHBoundProof2} in the interval $(s, t)$, with $r_1 \le s\le t \le r_0$ we get
\begin{equation}\label{eqn:QNDoublingdHProof}
    \frac{H(t)}{H(s)} \le \widetilde{C} \left( \frac{t}{s} \right)^C \quad \text{for all } r_1 \le s\le t \le r_0,
\end{equation}
and some constants $C = C(d, \delta_0)>0$ and $\widetilde C >0$ depending only on $d, \delta_0$ and $N(r_0)$. In particular, \eqref{eqn:QNDoublingdHProof} implies that $r_1 = 0$ and so $H(r)>0$ and  \cref{thm:QNfreqmon} applies for all $r\in(0,r_0]$.\medskip

Finally, integrating \eqref{eqn:QNDoublingdHProof} on $(0, r)$ for any $r \in (0, r_0)$ we obtain the doubling inequality
\begin{equation}\label{eqn:QNDoublingdL2Proof}
    \fint_{B_r} \varphi\left(\frac{\vert x \vert}{r} \right) u^2 \le \widetilde{C} \fint_{B_{r/2}} \varphi\left(\frac{\vert x \vert}{r}\right) u^2,
\end{equation}
for a constant $\widetilde C >0$ depending only on $d, \delta_0$ and $N(r_0)$, and passing to the limit as $\upsilon \to 1^-$ in \eqref{eqn:QNDoublingdL2Proof} (where $\upsilon$ is the parameter introduced in \eqref{eqn:QNCutoffDef}), we reach a contradiction with \eqref{eqn:QNvanishingAllOrdersHp}. Consequently, we have that
\begin{equation*}
    H(r) = 0 \quad \text{for all } r \in (0, R),
\end{equation*}
which concludes the proof of \eqref{eqn:uEquiv0BRProof} and thus of \cref{theorem:quasilinearUniqueContinuation}.
\qed


\section*{Acknowledgements} LS acknowledges the support of the NSF Career Grant DMS 2044954. LF and BV are supported by the European Research Council (ERC), under the European Union's Horizon 2020 research and innovation program, through the project ERC VAREG - {\em Variational approach to the regularity of the free boundaries} (No.\,853404). LF and BV acknowledge the MIUR Excellence Department Project awarded to the Department of Mathematics, University of Pisa, CUP I57G22000700001. LF is a member of INdAM-GNAMPA. BV acknowledges support from the projects PRA 2022 14 GeoDom (PRA 2022 - Università di Pisa) and MUR-PRIN ``NO3'' (No. 2022R537CS). We warmly thank Roberto Ognibene for the useful discussions on the existing literature about the unique continuation property for elliptic equations.


\bibliographystyle{plain}
\bibliography{FreeBoundary_bib.bib}

\medskip
\small
\begin{flushright}
\noindent 
\verb"lorenzo.ferreri@sns.it"\\
Classe di Scienze, Scuola Normale Superiore\\ 
piazza dei Cavalieri 7, 56126 Pisa (Italy)
\end{flushright}

\begin{flushright}
\noindent 
\verb"lspolaor@ucsd.edu"\\
Department of Mathematics, UC San Diego,\\
 AP\&M, La Jolla, California, 92093, USA
\end{flushright}

\begin{flushright}
\noindent 
\verb"bozhidar.velichkov@unipi.it"\\
Dipartimento di Matematica, Università di Pisa\\ 
largo Bruno Pontecorvo 5, 56127 Pisa (Italy)
\end{flushright}

\end{document}